\documentclass[11pt, a4paper,reqno]{amsart}
\usepackage[margin=1.3in]{geometry}
\usepackage{amsfonts, amsmath}
\usepackage[T1]{fontenc}
\usepackage{mathrsfs, enumitem}
\usepackage{hyperref}
\usepackage[utf8]{inputenc}
\usepackage{amssymb}
\usepackage{amsfonts}
\usepackage{amsmath}
\usepackage{amsthm,cancel}
\usepackage{color}
\usepackage{csquotes, tcolorbox}
\usepackage{tikz-cd}
\usepackage{lipsum}
\newtheorem{theorem}{Theorem}[section]
\newtheorem{lemma}[theorem]{Lemma}

\newtheorem{proposition}[theorem]{Proposition}
\newtheorem{corollary}[theorem]{Corollary}

\theoremstyle{definition}

\numberwithin{equation}{section}

\theoremstyle{remark}
\newtheorem{remark}[theorem]{Remark}
\theoremstyle{plain}
\newcommand{\ignore}[1]{}

\newcommand{\D}{\mathbb D}
\newcommand{\T}{\mathbb T}

\newcommand{\calA}{{\mathcal{A}^2}}
\newcommand{\calB}{{\mathcal{B}}}

\newcommand{\calD}{{\mathcal{D}_0}}

\newcommand{\calK}{{\mathcal{T}}(\mathcal{D}_0)}

\newcommand{\ip}[1]{\left\langle #1 \right\rangle}
\newcommand{\norm}[1]{\left\lVert #1 \right\rVert}
\newcommand{\abs}[1]{\left\lvert #1 \right\rvert}

\newcommand{\dou}{\partial}
\newcommand{\parz}[1]{\frac{\partial #1}{\partial z}}

\renewcommand{\le}{\leqslant}
\renewcommand{\ge}{\geqslant}
\renewcommand{\tilde}{\widetilde}
\title{Brown Halmos Operator Identity and Toeplitz Operators on the Dirichlet Space }

\author[A. Kujur]{Ashish Kujur}
\address[A. Kujur]{School of Mathematics \\
Indian Institute of Science Education and Research \\
Thiruvananthapuram}
\email{ashishkujur23@iisertvm.ac.in}

\author[M. R. Reza]{Md. Ramiz Reza}
\address[M. R. Reza]{School of Mathematics \\
Indian Institute of Science Education and Research \\
Thiruvananthapuram}
\email{ramiz@iisertvm.ac.in}
\thanks{The first author was supported through the Junior Research Fellowship by the Council of Scientific and Industrial Research (CSIR), India - Ref. No. 09/0997(18166)/2024-EMR-I}


\date{}

\begin{document}

\begin{abstract}
A well known result of Brown and Halmos shows that the Toeplitz operators induced by $L^{\infty}(\mathbb T)$ symbols on the Hardy space of the unit disc $\mathbb D$ are characterized by the operator identity $T_{\bar{z}}AT_z=A,$ where $T_z, T_{\bar{z}}$ are the Toeplitz operators induced by the function $z$ and $\bar{z}$ on the unit circle $\mathbb T$ respectively. In this paper we introduce and study a class of Toeplitz operators on the Dirichlet space $\calD$ induced by a symbol class $\mathcal T(\calD)= \overline{H^{\infty}_0(\mathbb D)} + \mathcal M(\calD),$ where $H^{\infty}_0(\mathbb D)$ denotes the set of all bounded analytic function on $\mathbb D$ vanishing at $0$ and $\mathcal M(\calD)$ denotes the multiplier algebra of the Dirichlet space $\mathcal D_0.$ We find that the Toeplitz operators on the Dirichlet space $\calD$ induced by the symbol class $\mathcal T(\calD)$ is completely characterized by the operator identity $T_{\bar{z}}AT_z=A.$  
\end{abstract}

\maketitle


\section{Introduction}
The theory of Toeplitz operators on  several Hilbert function spaces, for example, the Hardy space and the weighted Bergman spaces, is a well established subject. It has applications to and connection with many other areas, for example, function theory, prediction theory, singular integral equations, C* algebras and other areas of operator theory. Brown and Halmos had shown that a bounded operator $A$ acting on the Hardy space $H^2(\mathbb D)$ of the open unit disc $\mathbb D,$ is a Toeplitz operator $T_{g}$ induced by a unique symbol $g\in L^{\infty}(\mathbb T)$ if and only if $A$ satisfies the operator identity $T_{\bar{z}}AT_z=A,$  see \cite[Theorem 6]{BHT63}. It was well known that the Toeplitz operators with $L^{\infty}(\mathbb D)$ symbols on the classical Bergman spaces can not be characterized by such Brown Halmos type operator identity, see \cite{Englis88}. Nevertheless Louhichi and Olofsson characterized the Toeplitz operators induced by bounded harmonic symbols on the weighted Bergman spaces through an operator identity, see \cite{Olof08}. Later these results  has been extended to characterize Toeplitz operators with pluriharmonic symbol on the analytic Besov spaces on the unit ball in $\mathbb C^m,$ see \cite{Eschmier2019}. 

Toeplitz operators induced by various class of symbols on the Dirichlet space have been studied intensively in recent times, see for example \cite{Antti2017, Chen09, ChenNguyen10, DuiLee2004, Lee2007, Lee2009, LeeZhu2011, Lin2018, LuoXiao2018, RochbergWu1992, Wu1992, Yu2010} and the references therein. In this article we study the Toeplitz operators on the Dirichlet space induced by a class of symbols $\mathcal T(\calD)$ and find that the operator identity $T_{\bar{z}}AT_z=A$  completely characterize these class of Toeplitz operators among the set of all bounded operators on the Dirichlet space.

We start with some basic notations, a few definitions and some known results about Toeplitz operators on the Dirichlet space in order to introduce the symbol class $\mathcal T(\calD)$ and present our results. Let $dA=\frac{1}{\pi}dxdy$ be the normalized Lebesgue area measure on $\mathbb D.$ Let $L^2(\D)$ denotes the space of all complex valued measurable functions on $\mathbb D$ which are square integrable with respect to the Lebesgue area measure $dA$ and $L^{\infty}(\mathbb D)$ be the algebra of all essentially bounded measurable functions on $\mathbb D.$  The Sobolev space $W^{1,2}(\mathbb D)$ consists of all locally integrable complex valued functions $f$ on $\mathbb D$ for which  $f\in L^{2} \left( \D \right)$ and furthermore, the weak derivatives $\frac{\dou f}{\dou z}$ and $\frac{\dou f}{\dou \bar z}$ belong to $L^2(\D).$  It is well known that $W^{1,2}(\mathbb D)$ is a Hilbert space with respect to the inner product given by 
\begin{align*}
\Big\langle f,g \Big\rangle_{1,2}= \Big(\int_{\D} f \,dA\Big) \Big(\overline{\int_{\D} g \,dA} \Big)+ \Big\langle \frac{\dou f}{\dou z}, \frac{\dou g}{\dou z} \Big\rangle_{L^2(\D)}  + \Big\langle \frac{\dou f}{\dou \bar z}, \frac{\dou g}{\dou \bar z} \Big\rangle_{L^2(\D)},
\end{align*}
 see \cite[Chapter III]{Adams75} for more results on Sobolev space. The Dirichlet space $\mathcal D_0$ is the closed subspace of all holomorphic functions $f\in W^{1,2}(\mathbb D)$ with $f(0)=0.$  It is well known that $\mathcal D_0$ is a reproducing kernel Hilbert space with the kernel function given by
 \begin{align*}
 R_w(z)= \sum\limits_{k=1}^{\infty}\frac{z^k\bar{w}^k}{k}= \ln \frac{1}{(1-z\bar w)},\,\,\,z,w\in\mathbb D.
 \end{align*}
For any $\varphi \in W^{1,2}(\mathbb D),$ we have that 
$$\varphi|_{_\mathbb T}(e^{i\theta}):= \lim\limits_{r\to 1-}\varphi (re^{i\theta})$$ exists for almost all $\theta\in[0,2\pi],$ see \cite[Proposition 1]{ChenNguyen10}. Moreover $\varphi|_{_\mathbb T}\in L^1(\mathbb T)$ and $\Phi= P[\varphi|_{_\mathbb T}],$ the Poisson extension of $\varphi|_{_\mathbb T}$ on $\mathbb D,$ belongs to $W^{1,2}(\mathbb D),$ see \cite[Theorem 1]{ChenNguyen10}.
We let $\Lambda_0$ be the subspace of $W^{1,2}(\mathbb D)$ given by 
\begin{align*}
\Lambda_0 := \{ f\in W^{1,2}(\mathbb D) : f|_{_\mathbb T}=0 \}.
\end{align*}
Note that $\varphi-\Phi \in \Lambda_0$ and $\varphi = (\varphi -\Phi) +\Phi$ is an orthogonal decomposition of $\phi$ in $\Lambda_0$ and $\Lambda_0^{\perp}$ in $W^{1,2}(\mathbb D).$ Moreover the space $W^{1,2}(\mathbb D)$ decomposes as the following way
\begin{align*}
W^{1,2}(\mathbb D)= \Lambda_0  \oplus \overline{\mathcal D_0}\oplus \mathbb C \oplus \mathcal D_0,
\end{align*}
see \cite[Theorem 1]{ChenNguyen10}. Let $P_{\mathcal D_0}$ be the orthogonal projection from $W^{1,2}(\mathbb D)$ onto $\mathcal D_0.$ Using the reproducing property of the kernel function $R_w,$ one can see that $P_{\mathcal D_0}$ can be represented by the following integral formula:
\begin{align}\label{Projection formula}
(P_{\mathcal D_0}f)(w)= \Big\langle f,R_w \Big\rangle_{1,2}= \int_{\mathbb D} \Big ( \frac{\dou f (z)}{\dou z} \Big) \Big(\overline{\frac{\dou R_w (z)}{\dou z}}\Big)\,dA(z)  \Big),\,\, w\in\mathbb D,
\end{align} 
for every $f\in  W^{1,2}(\mathbb D).$ 
For a function $\varphi\in W^{1,2}(\mathbb D),$ the Toeplitz operator $T_{\varphi}$ with symbol $\varphi,$ is a densely defined operator on $\mathcal D_0$ given by 
\begin{align}\label{Toeplitz by projection}
T_{\varphi}(f)= P_{\mathcal D_0}(\varphi f),
\end{align}
for all $f\in $ Dom$(T_{\varphi})= \{g\in\calD: \varphi g \in W^{1,2}(\mathbb D)\}.$ 
In view of \eqref{Projection formula}, note that for any $\varphi\in W^{1,2}(\mathbb D)$ and  $f \in $ Dom$(T_{\varphi}),$ we have 
\begin{align}\label{Motivation for genral toeplitz}
    (T_{\varphi} (f)) (w) = \int_{\D} \left( \frac{\dou \left( \varphi f \right)}{\dou z}  \right) \overline{\left( \frac{\dou R_{w} (z)}{\dou z}\right)}\, dA(z), \qquad  w \in \D.
\end{align}
It is known that for any symbol $\varphi \in W^{1,2}(\mathbb D),$ the Toeplitz operator $T_{\varphi}$ defines a bounded operator on $\mathcal D_0$ if and only if $\Phi= P[\varphi|_{_\mathbb T}],$ the Poisson extension of $\varphi|_{_\mathbb T}$ on $\mathbb D,$ belongs $L^{\infty}(\mathbb D)$ and $\abs{\frac{\dou \Phi}{ \dou z}}^{2} dA$ is a \textit{Carleson measure} for $\mathcal D_0$ i.e. there exists a constant $C>0$ such that
\begin{align*}
 \int_{\D} \abs{f}^{2} \abs{\frac{\dou \Phi }{\dou z}} ^{2} dA(z)  \le C \norm{f}_{\mathcal D_0}^{2},\,\,\,f\in\mathcal D_0,
\end{align*}
see \cite[Theorem 2.2]{LuoXiao2018}. It is also known that for any symbol $\varphi \in W^{1,2}(\mathbb D),$ if the Toeplitz operator $T_{\varphi}$ defines a bounded operator on $\mathcal D_0,$ then it turns out that the operator $T_{\varphi}$ is determined by the boundary function $\varphi|_{_\mathbb T}$ and $T_{\varphi}= T_{\Phi},$ see \cite[Theorem 2.2]{LuoXiao2018}. This motivates us to consider the  symbol class $\mathcal T(\calD)$ defined by 
\begin{align*}
\mathcal T(\calD) := \Big\{\psi \in h^{\infty}(\mathbb D) : \abs{\frac{\dou \psi}{ \dou z}}^{2} dA \,\,\mbox{is a Carleson measure for\,\,} \mathcal D_0\Big\},
\end{align*}
where $h^{\infty}(\mathbb D)$ denotes the space of all complex valued bounded harmonic functions on $\mathbb D.$ We would like to point out that it is not the case that $\mathcal T(\calD) \subseteq W^{1,2} \left( \D \right)$, see Remark \ref{remark about symbol class} of this article.  Following \cite[Sec 2.2]{LuoXiao2018}, we use the weaker expression in \eqref{Motivation for genral toeplitz} to define a Toeplitz operator $T_{\psi}$ for $\psi \in \mathcal T(\calD).$  The Toeplitz operator $T_{\psi}$ with symbol $\psi$ in  $\mathcal T(\calD)$ is defined by
\begin{equation*}
    (T_{\psi} (f)) (w) := \int_{\D} \left( \frac{\dou \left( \psi f \right)}{\dou z}  \right) \overline{\left( \frac{\dou R_{w} (z)}{\dou z}\right)} \, dA(z), \qquad  w \in \D, f \in \mathcal D_0.
\end{equation*}
where $R_{w}(z)= \sum_{n=1}^{\infty} \frac{z^n\bar{w}^n}{n}$ is the reproducing kernel of $\mathcal D_0$ at the point $w \in \D$.  In this article we first show that $T_{\psi}$ defines a bounded operator on $\mathcal D_0$ for every $\psi\in \mathcal T(\calD).$ We find that the operator $T_{\psi}$ on the Dirichlet space $\mathcal D_0$ satisfies the Brown Halmos type identity $T_{\bar{z}}T_{\psi}T_z= T_{\psi}$ for every $\psi\in \mathcal T(\calD),$ see Corollary \ref{toeplitz condition satisfies}. The main goal of this article is to obtain a characterization the Toeplitz operators $\{T_{\psi} : \psi\in \mathcal T(\calD)\}$ through the Brown-Halmos type operator identity.  
\begin{theorem}\label{thm:BHT theorem for D}
 If $T\in \mathcal B(\mathcal D_0)$ satisfies the identity $T_{\bar{z}}TT_z= T$, then there exists a unique $\psi\in \mathcal T(\calD)$ such that $T=T_{\psi}.$ 
\end{theorem}
We would like to point out that although the Toeplitz operator $T_{\bar{z}}$ coincides with $T_z^*,$ the adjoint of $T_z,$ on the Hardy space $H^2(\mathbb D)$ but on the Dirichlet space $\mathcal D_0$, $T_{\bar{z}} \neq T_z^*.$ As a corollary to Theorem \ref{thm:BHT theorem for D}, we obtain that the subspace $\{T_{\psi}: \psi\in \mathcal T(\calD)\}$ is a closed subspace in the weak operator topology (W.O.T) of $\mathcal B(\mathcal D_0).$ This suggests that the symbol class $\mathcal T(\calD)\}$ is the largest class (in some sense) for the representation of the bounded Toeplitz operators on the Dirichlet space $\calD$. We further find that 
\begin{align*}
    \mathcal T(\calD)= \mathcal M(\calD) + \overline{H^{\infty}_0(\mathbb D)},
\end{align*}
where $\mathcal M(\calD)$ denotes the multiplier algebra of the Dirichlet space $\calD$ and $\overline{H^{\infty}_0(\mathbb D)}$ is given by $\overline{H^{\infty}_0(\D)}= \{ \bar{f}: f\in H^{\infty}(\D), f(0)=0\}$, see Proposition \ref{Toeplitz symbol decomposition}.

Investigating the compactness of the Toeplitz operators on function spaces has been an active area of research, for example see \cite{Lee2007,LuoXiao2018} for the characterizations of the compact Toeplitz operators $T_{\varphi}$ on the Dirichlet space $\calD$ for $\varphi$ coming from various symbol class. We find that the same characterization remains valid for the symbol class $\mathcal T(\calD),$ as expected. In particular we find that for any $\varphi\in\mathcal T(\calD),$ the operator  $T_{\varphi}$ is compact if and only if $\varphi=0,$ see Proposition \ref{Compact toeplitz}.

In this paper we also discuss several algebraic properties of the class of Toeplitz operators $\{T_{\varphi} : \varphi\in \mathcal T(\calD)\}$ on the Dirichlet space $\calD.$ The corresponding problem on the Toeplitz operators acting on the Hardy, weighted Bergman spaces or harmonic Bergman spaces has been well studied by several authors, see for example \cite{Axler91,BHT63,Choe99,Cuko94} and references therein. For a symbol $\varphi\in \mathcal T(\calD),$  the matrix representation of the Toeplitz  operator $T_{\varphi}$  with respect to the orthogonal basis $\{z^n: n\in \mathbb N\}$ of $\calD$ turns out to be the \textit{Toeplitz matrix} $(\!(\hat{\varphi} (i-j))\!)_{i,j=1}^{\infty},$ where the symbol $\varphi$ is given by
    \begin{equation*}
	\varphi(z) = \sum_{k\ge 0} \hat{\varphi} (k) z^{k} + \sum_{k>0} \hat{\varphi} (-k) \overline{z}^{k}, \qquad  z\in \D.
    \end{equation*}
Consequently, characterizations of several algebraic properties of the  Toeplitz operators $\{T_{\varphi} : \varphi\in \mathcal T(\calD)\}$ remains similar to that of the Toeplitz operators $\{T_{\varphi} :\varphi \in L^{\infty}(\mathbb T)\}$ on the Hardy space $H^2(\mathbb D).$  In \cite{Lee2007}, the author studied various algebraic properties  for the class of Toeplitz operators  $\{T_{\varphi} : \varphi\in \Omega \cap h(\D)\},$ where the symbol class $\Omega$ is given by $\Omega=\{ \varphi \in \mathscr C^1(\mathbb D):  \varphi, \frac{\dou \varphi }{\dou z}, \frac{\dou \varphi }{\dou \bar z} \in L^{\infty}(\mathbb D)\}$ and $h(\D)$ denotes the set of all complex valued harmonic functions on $\D$. It is straightforward to see that $\Omega \cap h(\D)\subseteq \mathcal T(\calD).$ We find that each of the characterizations of the algebraic properties for the symbol class $\Omega \cap h(\D)$ as mentioned in \cite{Lee2007}, remains valid for the larger symbol class $\mathcal T(\calD)$ as well, see Section \ref{Sec 4} for more detailed discussion.

The organization of the paper is as follows. In Section 2, we study the boundedness of the Toeplitz operators  $\{T_{\varphi} : \varphi\in \mathcal T(\calD)\}$ on the Dirichlet space $\calD$  and obtain the decomposition of $\mathcal T(\calD)$ into sum of $\mathcal M(\mathcal D_0)$ and $\overline{H^{\infty}_0(\mathbb D)}.$
In Section 3, we present the main result of this article, namely, we show that the operator identity $T_{\bar{z}}AT_z=A$ characterize the Toeplitz operators $\{T_{\varphi} : \varphi\in \mathcal T(\calD)\}$ in $\mathcal B(\calD).$ In Section 4, we discuss various characterizations of algebraic properties of these Toeplitz operators.


\section{Toeplitz Operators with symbols in $\mathcal T(\calD)$ }
Let $\mathcal A^2$ denotes the Bergman space of the open unit disc $\mathbb D$ and $K^B_w(z)=\frac{1}{(1-\bar{w}z)^2},$ denotes the reproducing kernel of $\calA$ at the point $w \in \D$ so that
\begin{align*}
\ip{g,K^B_w}_{\calA}=g(w),\,\,g\in\mathcal A^2,\,w\in\mathbb D.
\end{align*}
We use the symbol $P_{\calA}$ to denote the orthogonal projection from $L^2(\mathbb D)$ onto $\calA.$ Using the reproducing property, we have 
\begin{align*}
\left(P_{\calA}f\right)(w)= \ip{f,K^B_w}_{\calA},\qquad f\in L^2(\mathbb D), w\in\mathbb D.
\end{align*}
Let $R_{w}(z)= \sum_{n=1}^{\infty} \frac{z^n\bar{w}^n}{n} =\ln\frac{1}{(1-\bar{w}z)},$ denotes the reproducing kernel of the Dirichlet space $\mathcal D_0$ at the point $w \in \D$ so that
\begin{align*}
\ip{f,R_w}_{\mathcal D_0}=f(w),\,\,f\in\mathcal D_0,\,w\in\mathbb D.
\end{align*}
We would like to note down an important identity involving the reproducing kernel function $R_w$ and $K^B_w,$ as this will be useful for us at later stage.  
\begin{align}\label{Bergman and Dirichlet identity}
\frac{\dou}{\dou w}\overline{\left( \frac{\dou R_{w} (z)}{\dou z}\right)} = \frac{\dou}{\dou w}\left( \frac{w}{1-w\bar z} \right)
=\frac{1}{(1-w\bar{z})^2} = \overline{K^B_w(z)},\,\,w\in\mathbb D.
\end{align}   
Now we introduce a densely defined Toeplitz operator induced by a symbol in the class $\tilde{\mathcal T}(\calD)$ which is defined by 
$$\tilde{\mathcal T}(\calD):= \Big\{\varphi\in h^{\infty}(\mathbb D): \frac{\dou \varphi}{\dou z} \in \mathcal A^2\Big\},$$
where $h^{\infty}(\mathbb D)$ denotes the set of all bounded harmonic functions on unit disc $\mathbb D.$  For $\varphi \in \tilde{\mathcal T}(\calD)$, we consider the densely defined linear operator $T_{\varphi}$ on $\mathcal D_0$ given by 
\begin{equation*}
    (T_{\varphi} (f)) (w) := \int_{\D} \left( \frac{\dou \left( \varphi f \right)}{\dou z}  \right) \overline{\left( \frac{\dou R_{w} (z)}{\dou z}\right)} \, dA(z), \qquad   w \in \D, \,f\in \mbox{dom}(T_{\varphi}),
\end{equation*}
where dom$(T_{\varphi})=\{f\in \mathcal D_0: T_{\varphi}(f)\in\mathcal D_0\}.$ Suppose $\varphi\in \tilde{\mathcal T}(\calD).$ For any $w\in\mathbb D$ and $f\in $ span$\{z^n: n\geqslant 1\},$ note that
\begin{align*}
\left( T_{\varphi} (f) \right)' (w) &=  \frac{\dou}{\dou w} \left( \int_{\D} \left( \frac{\dou \varphi}{\dou z} f + f' \varphi \right) \left( \frac{w}{1-w\bar z} \right) dA(z)\right)  \\
&= \int_{\D} f(z) \frac{\dou \varphi}{\dou z} (z) \overline{K_{w}^{B} (z)} + \int_{\D} f'(z) \varphi (z) \overline{K_{w}^{B}(z)}\, dA(z) \\
&= f(w) \frac{\dou \varphi}{\dou z} (w) + P_{\calA}(\varphi f')(w). 
\end{align*}
Thus we have 
\begin{align}\label{action of densely defined toeplitz op}
\left( T_{\varphi} (f) \right)'= \frac{\dou \varphi}{\dou z} f + P_{\calA}(\varphi f'),\,\,\mbox{for every}\,\,f\in \,\mbox{span}\{z^n: n\geqslant 1\}.
\end{align}
This gives us that span$\{z^n: n\geqslant 1\} \subseteq \mbox{dom}(T_{\varphi})$ for any $\varphi\in \tilde{\mathcal T}(\calD).$ Thus $T_{\varphi}$ is densely defined operator on $\mathcal D_0$ for every $\varphi\in \tilde{\mathcal T}(\calD).$ Further note that for any $\varphi\in \tilde{\mathcal T}(\calD),$ we have
\begin{align}\label{lem:inner-product}
\ip{T_{\varphi} (f),g}_{\mathcal D_0}
&= \ip{\frac{\dou (\varphi f)}{\dou z}, g'}_{L^2(\D)},\qquad f,g\in \,\mbox{span}\{z^n: n\geqslant 1\}.
\end{align}

Now we provide a characterization for when $T_{\varphi}$ defines a bounded operator on $\mathcal D_0.$  This characterization is already known when $\varphi$ is in the symbol class $W^{1,2}(\mathbb D),$ see \cite[Theorem 2.2]{LuoXiao2018}.
\begin{theorem}\label{thm:bdd-top-op}
Let $\varphi \in \tilde{\mathcal T}(\calD).$ Then the following are equivalent: 
    \begin{itemize}
	\item[(a)] $T_{\varphi}$ extends to a bounded operator on $\mathcal D_0$.
	\item[(b)] The measure $\abs{\frac{\dou \Phi}{ \dou z}}^{2} dA$ is a Carleson measure for $\mathcal D_0,$ that is, there exists a constant $C>0$ such that
\begin{align*}
 \int_{\D} \abs{f}^{2} \abs{\frac{\dou \Phi }{\dou z}} ^{2} dA  \le C \norm{f}_{\mathcal D_0}^{2},\,\,\,f\in\mathcal D_0.
\end{align*}
    \end{itemize}
 \end{theorem}
\begin{proof}
First assume that $T_{\varphi}$ extends to a bounded operator on $\mathcal D_0$. In view of the equation \eqref{action of densely defined toeplitz op} it follows that for any $f\in \mbox{span}\{z^n:n\geqslant 1\},$
\begin{align*}
	\int_{\D} \abs{\frac{\dou \varphi}{\dou z} f}^{2} dA &=\int_{\D} \abs{\frac{\dou \varphi}{\dou z} f  + P_{\calA}(\varphi f')  - P_{\calA}(\varphi f') }^{2} dA(z) \\
	&\le 2\int_{\D} \abs{\frac{\dou \varphi}{\dou z} f  + P_{\calA}(\varphi f')}^{2} dA +2 \int_{\D}  \abs {P_{\calA}(\varphi f') }^{2} dA(z) \\
	&=  2\norm{(T_{\varphi}f)'}_{\mathcal A^2}^{2} + 2 \norm{P_{\calA}(\varphi f')}_{\mathcal A^2}^{2}\\
	& \le 2\norm{T_{\varphi}f}_{\mathcal D_0}^{2} + 2 \norm{\varphi}_{\infty}^{2} \norm{f'}_{\mathcal A^2}^{2} \\
	&\le 2\left( \norm{T_{\varphi}}^{2} + \norm{\varphi}_{\infty} ^{2} \right) \norm{f}_{\mathcal D_0}^{2}.
    \end{align*}
Note that the point evaluation at every point $w\in\mathbb D,$ is a continuous linear functional on $\mathcal D_0$ and span$\{z^n: n\geqslant 1\}$ is a dense subset in $\mathcal D_0.$ So, an application of Fatou's lemma will give us 
\begin{align*}
\int_{\D} \abs{f}^{2} \abs{\frac{\dou \varphi }{\dou z}} ^{2} dA  \le 2\left( \norm{T_{\varphi}}^{2} + \norm{\varphi}_{\infty} ^{2} \right) \norm{f}_{\mathcal D_0}^{2},\,\,\,f\in\mathcal D_0.
\end{align*}
Now assume that there exists a $C>0$ such that 
\begin{align*}
 \int_{\D} \abs{f}^{2} \abs{\frac{\dou \varphi }{\dou z}} ^{2} dA  \le C \norm{f}_{\mathcal D_0}^{2},\,\,\,f\in\mathcal D_0.
\end{align*}
Note that for every $f\in $ span$\{z^n: n\geqslant 1\},$ we have
\begin{align*}
	\norm{T_{\varphi} f}_{\mathcal D_0}^{2} &= \int_{\D} \abs{(T_{\varphi} f)' (w)}^{2} dA(w) \\
	&= \int_{\D} \abs{f(w) \frac{\dou \varphi}{\dou z} (w) + P_{\calA}(\varphi f')(w)}^{2} dA(w) \\
	&\le 2 \int_{\D} \Big(\abs{f(w) \frac{\dou\varphi}{\dou z} (w)}^{2} + \abs{P_{\calA}(\varphi f')(w) }^{2}\Big) dA(w) \\
	&\le 2C \norm{f}_{\mathcal D_0}^{2} + 2\norm{P_{\calA}(\varphi f')}_{\calA}^2\\
	&\le 2C \norm{f}_{\mathcal D_0}^{2} + 2\norm{\varphi}_{\infty}^2 \norm{f'}_{\mathcal A^2}^{2}\\
	& \le 2\left( C + \norm{\varphi}_{\infty}^2 \right) \norm{f}_{\mathcal D_0}^{2} \stepcounter{equation}\tag{\theequation}\label{eqn:operator-bound}\\
    \end{align*}
  It is straightforward to see that for any $w\in\calD$ the map $f\mapsto f'(w)$ is also bounded linear functional on $\mathcal D_0.$  Again using density of span$\{z^n: n\geqslant 1\}$  in $\mathcal D_0$ and an application of Fatou's lemma will give us
    \begin{align*}
    \norm{T_{\varphi} f}_{\mathcal D_0}^{2} \leqslant  2\left( C + \norm{\varphi}_{\infty}^2 \right) \norm{f}_{\mathcal D_0}^{2},\,\,\,f\in \mathcal D_0,
    \end{align*}
 which shows that $T_{\varphi}$ defines a bounded operator on $\mathcal D_0.$ This completes the proof.   
    \end{proof} 
 Recall that the symbol class $\calK$ is defined as
 \begin{align*}
 \calK  = \Big\{\varphi \in h^{\infty}(\mathbb D) : \abs{\frac{\dou \varphi}{ \dou z}}^{2} dA \,\,\mbox{is a Carleson measure for\,\,} \mathcal D_0\Big\}.
 \end{align*}
 Note that for any $\varphi \in\calK,$ we have that $z\frac{\dou \varphi}{\dou z} \in \calA.$ It follows that $\frac{\dou \varphi}{\dou z} \in \calA$ for every $\varphi \in\calK.$ Thus in view of Theorem \ref{thm:bdd-top-op} the following corollary is immediate.
 \begin{corollary}\label{bounded vs k tilde}
 $\calK= \{\varphi \in \tilde{\mathcal T}(\calD): T_{\varphi}$ defines a bounded operator on $\mathcal D_0\}.$
 \end{corollary}
 \begin{remark}\label{remark about k and k tilde}
It turns out that $\calK \subsetneq \tilde{\mathcal T}(\calD).$ It is well known that $\mathcal M(\mathcal D_0)\subsetneq \left(\mathcal D_0 \cap H^{\infty}(\D)\right).$ Take any $g\in \left(\mathcal D_0 \cap H^{\infty}(\D)\right) \setminus \mathcal M(\mathcal D_0).$ Then $g\in \tilde{\mathcal T}(\calD) \setminus \calK.$
\end{remark}
 
\begin{remark}\label{remark about symbol class}
If $\varphi$ is a harmonic function in  $W^{1,2} \left( \D \right)$ such that $T_{\varphi}$ defines a bounded operator on $\mathcal D_0,$ then $\varphi\in \calK,$ see \cite[Theorem 2.2]{LuoXiao2018}. Nevertheless it is not the case that $\calK \subseteq W^{1,2} \left( \D \right)$. To see this, take any $f\in H^{\infty} \left( \D \right)$ such that $\norm{\frac{\dou f}{\dou z}}_{L^{2} \left( \D \right)} = \infty$ (for instance, one can take $f$ to be some infinite Blaschke product) and consider $\Psi = \overline{f}$. Then $\Psi \in \calK$ because $\abs{\frac{\dou \Psi}{\dou z}}^{2} dA(z)$ is the zero measure and $\Psi$ is a bounded harmonic function. But it is clear that $\Psi \not\in W^{1,2} \left( \D \right)$ because $\norm{\frac{\dou \Psi}{\dou \bar z}}_{L^{2} \left( \D \right)} = \infty.$ 
\end{remark}

\begin{proposition}\label{Toeplitz matrix repn}
    Let $\varphi$ be a function in $\tilde{\mathcal T}(\calD)$ and suppose that
    \begin{equation*}
	\varphi(z) = \sum_{k\ge 0} \hat{\varphi} (k) z^{k} + \sum_{k>0} \hat{\varphi} (-k) \overline{z}^{k}, \qquad  z\in \D.
    \end{equation*}
Then we have that
\begin{equation*}
    T_{\varphi} (z^{n})= \sum_{k \ge 1} \hat{\varphi} \left( k-n \right) z^{k}, \qquad n \ge 1 .
\end{equation*} Consequently, it follows that the matrix representation of $T_{\varphi}$ with respect to the orthogonal basis $\left\{ z^{j} : j \in \mathbb Z_{\geqslant 1} \right\}$ of $\mathcal D_0$ is of the following {\textit{Toeplitz}} form:
    \begin{equation*}\label{thm:action-of-toeplitz-operator}
  \left[ T_{\varphi} \right]=	\begin{pmatrix}
	    \hat{\varphi} \left( 0 \right) & \hat{\varphi} \left( -1 \right) & \hat{\varphi} \left( -2 \right) & \cdots \\
\hat{\varphi} \left( 1 \right) & \hat{\varphi} \left( 0 \right) & \hat{\varphi} \left( -1 \right) & \cdots \\
\hat{\varphi} \left( 2 \right) & \hat{\varphi} \left( 1 \right) & \hat{\varphi} \left( 0 \right) & \ddots \\
\vdots & \vdots & \ddots & \ddots
	\end{pmatrix}.
    \end{equation*}
\end{proposition}
\begin{proof}
First note that as $\frac{\dou \varphi}{ \dou z}\in\calA,$ we have that 
    \begin{align}\label{Convergence criterion}
    \sum\limits_{k=1}^{\infty} k |\hat{\varphi} (k)|^2 <\infty.
    \end{align}
For any $z\in \D$ and $n \geqslant 1$, we get 
    \begin{equation*}
	\varphi(z)z^{n} = \sum_{k\geqslant 0} \hat{\varphi} (k) z^{n+k} + \sum_{k>0} \hat{\varphi} (-k) \overline{z}^{k}z^{n}.
    \end{equation*}
   
     and consequently, we have 
     \begin{equation*}
	 \frac{\dou}{\dou z}\left( \varphi(z)z^{n} \right) = \sum_{k\ge 0} (n+k) \hat{\varphi} (k) z^{n+k-1} + \sum_{k>0} n \hat{\varphi} (-k) \overline{z}^{k}z^{n-1}.
     \end{equation*}

Now note that for each $w\in\mathbb D,$
     \begin{align*}
	 \ip{ \sum_{k\ge 0} (n+k) \hat{\varphi} (k) z^{n+k-1}, \sum_{j\ge 0} \overline{w}^{j+1} z^{j} }_{L^{2} \left( \D \right)} &= \sum_{j,k \ge 0} \left( n+k \right) \hat{\varphi} (k) w^{j+1} \ip{z^{n+k-1}, z^{j}}_{L^{2} \left( \D \right)} \\
	 &= \sum_{j,k \ge 0} \hat{\varphi} \left( k \right) \delta_{n+k-1, j} w^{j+1},
	      \end{align*}
      and we have
      \begin{align*}
	  \ip{\sum_{k>0} n \hat{\varphi} (-k) \overline{z}^{k}z^{n-1}, \sum_{j\ge 0} \overline{w}^{j+1} z^{j}}_{L^{2} \left( \D \right)} &= \sum_{j\ge 0 , k > 0} n \hat{\varphi} \left( -k \right) w^{j+1} \ip{\overline{z}^{k} z^{n-1} , z^{j}}_{L^{2} \left( \D \right)} \\
	  &= \sum_{j\ge 0, k> 0} \hat{\varphi} \left( -k \right) \delta_{n-1, j+k} w^{j+1}
      \end{align*}
Hence, we find that for each $w\in\mathbb D,$
      \begin{align*}
	  (T_{\varphi} (z^{n}))(w) &= \ip{ \frac{\dou}{\dou z}\left( \varphi(z)z^{n} \right), \sum_{j \ge 0} \overline{w}^{j+1}z^{j} }_{L^{2} \left( \D \right)} 
	  = \sum_{j\ge 1} \hat{\varphi} \left( j-n \right) w^{j}.
      \end{align*}
In view of \eqref{Convergence criterion}, we have that $\sum_{j\geqslant 1} j|\hat{\varphi}(j-n)|^2.$ Thus $\sum_{j\geqslant 1} \hat{\varphi} \left( j-n \right) z^{j}\in \mathcal D_0.$ Since the evaluation at every point $w\in\mathbb D$ is a bounded linear functional on $\mathcal D_0,$ it follows that 
      \begin{align*}
       T_{\varphi} (z^{n})=\sum_{j\ge 1} \hat{\varphi} \left( j-n \right) z^{j},
      \end{align*}
where the above sum converges in the norm in $\calD.$ This gives us the matrix representation of $T_{\varphi}$ is in the required form. 
\end{proof}
In view of the Proposition \ref{Toeplitz matrix repn}, the following corollary is now immediate. 
\begin{corollary}
For any $\varphi\in\tilde{\mathcal T}(\calD)$ the operator $T_{\varphi}=0$ if and only if $\varphi =0.$
\end{corollary}
Proposition \ref{Toeplitz matrix repn} also gives us the answer of zero product problem for the class of Toeplitz operators $\{T_{\varphi}: \varphi\in \mathcal T (\calD)\}$  acting on the Dirichlet space $\calD$. The classical zero product problem asks when a product of finitely many bounded Toeplitz operators on the Hardy space is equal to zero. This has been resolved completely by Aleman and Vukoti\'{c}, see \cite{AVZP2009}. Using their result,  Y. J. Lee proved that the zero product problem for the Toeplitz operators on $\calD$ has the same answer provided the the symbols associated to the Toeplitz operators are coming from the class $\{ \varphi \in h(\mathbb D):\varphi,\frac{\dou \varphi }{\dou z}, \frac{\dou \varphi }{\dou \bar z} \in L^{\infty}(\mathbb D) \},$  see \cite[Theorem 1.2]{Lee2009}. In view of Proposition \ref{Toeplitz matrix repn}, Lee's argument can be imitated to obtain the same result for the relatively larger class of Toeplitz operators on $\calD$ induced by the symbol class $\mathcal T(\calD).$ Therefore we omit the proof of the following Proposition and refer the reader to \cite[Theorem 1.2]{Lee2009} for the proof. 
\begin{proposition}
Let $\varphi_1,\ldots,\varphi_k\in \mathcal T(\calD).$  Then $T_{\varphi_1}\cdots T_{\varphi_k}=0$ on $\calD$ if and only if $\varphi_j =0$ for some $j=1,\ldots,k.$
\end{proposition}

Now we discuss the problem of when a Toeplitz operator $T_{\varphi},$ acting on $\calD,$ is compact for $\varphi\in \mathcal T(\calD).$ It is well known that a Toeplitz operator $T_{\varphi},$  with $\varphi \in L^{\infty}(\mathbb T),$ acting on the Hardy space  is compact if and only if $\varphi=0.$ Several authors has studied the compactness problem for Toeplitz operators on $\calD$ induced by various class of symbols, see for example \cite{Lee2007},\cite{LuoXiao2018}. However, we offer an alternative proof using the \textit{Toeplitz matrix} nature of the matrix representation of the Toeplitz operators on the Dirichlet space $\calD.$
\begin{proposition}\label{Compact toeplitz}
    Let $\phi \in \calK$. Then $T_\varphi$ is compact if and only if $\varphi = 0$.
\end{proposition}
\begin{proof} 
Let $\varphi \in \calK$. We write
\begin{equation*}
    \varphi (z) = \sum_{k \ge 0} c_k z^k + \sum_{k \ge 1} c_{-k} \overline{z}^k, \qquad z \in \D.
\end{equation*}
Fix a  $k \ge 0$. By Proposition \ref{Toeplitz matrix repn}, we have that for any $m\geqslant 1,$
\begin{equation*}
    c_k = \frac{1}{\norm{z^m}^2} \ip{T_{\varphi}z^{m+k}, z^m}.
\end{equation*}
Now assume that $T_\varphi$ is compact. Since $\{z^m: m\geqslant 1\}$ is an orthogonal basis of $\mathcal D_0,$ we have that $T_\varphi ^* \left( \frac{z^m}{\norm{z^m}} \right)$ converges to $0$ in $\calD$. Hence, we have that for each $m \ge 1$,
\begin{align*}
    \abs{c_k} &= \frac{1}{\norm{z^m}^2} \abs{\ip{T_{\varphi}z^{m+k}, z^m}} 
    \leqslant \sqrt{1+ \frac{k}{m}} \norm{T_\varphi ^* \left( \frac{z^m}{\norm{z^m}} \right)}.
\end{align*} 
Now, letting $m \to \infty$, we have that $c_k = 0$ for each $k \ge 0$. A similar argument shows that $c_{-k} = 0$ for all $k \ge 1$. Thus, we have that $\varphi = 0$. 
\end{proof}

Let $\mathcal M(\mathcal D_0)$ be the multiplier algebra of $\calD$, that is,
\begin{align*}
    \mathcal M(\mathcal D_0)= \{\psi \in \mathcal O(\D): \psi f \in \calD \mbox{\;\;for every\,\,} f\in\calD\}.
\end{align*} It is straightforward to verify that 
\begin{align*}
    \mathcal M(\mathcal D_0)= \{\psi \in H^{\infty}(\D): |\psi'(z)|^2dA(z)\, \mbox{is a Carleson measure for }\,\calD\},
\end{align*} 
see \cite[Theorem 5.1.7]{Primer} for a similar argument of the proof. Let $H^{\infty}_0(\D)$ denotes the set of all bounded holomorphic function on $\D$ which vanishes at $0,$ i.e.
\begin{align*}
    H^{\infty}_0(\D)= \{f\in H^{\infty}(\D): f(0)=0\}.
\end{align*}
We conclude this section after showing that $ \mathcal T(\calD)= \mathcal M(\calD) + \overline{H^{\infty}_0(\mathbb D)}$ and both the holomorphic and antiholomorphic part of $\varphi\in \calK$ are in $\calK.$ 
\begin{proposition}\label{Toeplitz symbol decomposition}
 Let $\varphi$ be a harmonic function on $\mathbb D$ and suppose that 
    \begin{equation*}
	\varphi(z) = \varphi_1(z) +\varphi_2(z) , \,\mbox{where\,\,}\,\varphi_1(z) = \sum_{k\ge 0} \hat{\varphi} (k) z^{k} , \varphi_2= \sum_{k>0} \hat{\varphi} (-k) \overline{z}^{k}, \,\, z\in \D.
    \end{equation*} 
    Then $\varphi\in \mathcal T(\calD)$ if and only if $\varphi_1 \in \mathcal M(\mathcal D_0)$ and $\varphi_2\in \overline{H^{\infty}_0}.$ In particular both $\varphi_1$ and $\varphi_2$ are in $\calK.$
    \label{prop:decomposition-holomorphic-antiholomorphic}
\end{proposition}
\begin{proof}
    Since $\varphi\in\calK,$ the map $T: \calD \to \calA$ defined by $T(f)= \varphi_1'f,\,f\in\calD,$ is a bounded linear transformation. It follows that $T^*(K^B_w)= 
    \overline{\varphi_1'(w} R_w$ for every $w\in\mathbb D.$ Thus we obtain that 
     \begin{align*}
        |\varphi_1'(w)|^2 \leqslant \|T\|^2 \frac{R_w(w)}{K^B_w(w)} = \|T\|^2 (1-|w|^2)^2 \ln \left(\frac{1}{1-|w|^2}\right),\qquad w\in \D.
    \end{align*}
    It follows that 
    \begin{align*}
        \lim\limits_{|w|\to 1-} \frac{|\varphi_1'(w)|^2}{1-|w|^2}=0.
    \end{align*}
    This gives us that $\varphi_1'$ is bounded and hence $\varphi_1$ is bounded. It follows that both  $\varphi_1$ and $\varphi_2$ are bounded and in $\calK.$
\end{proof}

\section{Brown-Halmos operator identity and Toeplitz operators}
Let $\mathbb C[z,\bar{z}]$ denotes the set of all trigonometric  polynomials in $z$ and $\bar{z},$ that is, 
 \begin{align*}
     \mathbb C[z,\bar z]:= \Big\{ \sum\limits_{j=0}^n c_{-j}\bar{z}^j + \sum\limits_{j=1}^nc_jz^j,\,c_j\in\mathbb C\Big\}.
 \end{align*}
 Since  the Lebesgue measure $dA$ on unit disc $\mathbb D$ is a Carleson measure for $\mathcal D_0,$ it follows that $\mathbb C[z,\bar{z}] \subseteq \mathcal T(\calD).$ Consider the operator $T_z$ and $T_{\bar{z}}.$ Let $f\in\mathcal D_0$ having power series representation of the form $f(w)=\sum_{j\geqslant 1}a_jw^j,$ $w\in\mathbb D.$ Then 
\begin{align*}
T_z(f)(w)= \sum\limits_{j\geqslant 2}a_{j-1}w^{j},\,\,\,T_{\bar{z}}(f)(w)= \sum\limits_{j\geqslant 1}a_{j+1}w^{j},\,\,w\in\mathbb D.
\end{align*}
Thus w.r.t the orthogonal basis $\mathcal B=\{z^j:j\in\mathbb Z_{\geqslant 1}\}$ of $\mathcal D_0$, the operator $T_z$ and $T_{\bar{z}}$ represents the forward shift and the backward shift operator on $\mathcal D_0$ respectively. Moreover note that $T_{\bar{z}}T_z=I.$ In this section, we will show that the Brown Halmos type operator identity $ T_{\bar{z}}AT_z=A$ completely characterize the Toeplitz operators on $\mathcal D_0$ induced by the symbols in $\mathcal T(\calD).$ First, we provide a characterization of every operator $T\in  \mathcal B(\mathcal D_0)$ which satisfies the Brown Halmos type identity  $T_{\bar{z}}TT_z= T$ in terms of their matrix representation. 
\begin{proposition}\label{BHT implies toeplitz}
Let $T\in \mathcal B(\mathcal D_0).$ Suppose $\left[T\right]$ denotes the matrix representation of $T$ w.r.t the orthogonal basis $\mathcal B=\{z^j:j\in\mathbb Z_{\geqslant 1}\}$ of $\mathcal D_0,$ that is,
\begin{align*}
 Tz^{j} = \sum_{i \ge 1} [T]_{i,j} z^{i}, \qquad j\ge 1.
\end{align*}
Then $T$ satisfies the identity $T_{\bar{z}}TT_z= T$ if and only if 
\begin{align*}
[T]_{i,j}= [T]_{i+1,j+1},\,\,\mbox{for every}\,\,i,j\geqslant 1.
\end{align*}
\end{proposition}
\begin{proof}
It is straightforward to verify that for every $j\geqslant 1,$
\begin{align*}
T_{\bar{z}}TT_z(z^j)&= T_{\bar{z}} \left(\sum_{i \ge 1} [T]_{i,j+1} z^{i}\right)
= \sum_{i \ge 1} [T]_{i+1,j+1} z^{i}.
\end{align*}
Since $\{z^j:j\in\mathbb Z_{\geqslant 1}\}$ is an orthogonal basis of $\mathcal D_0,$ it follows immediately that $T_{\bar{z}}TT_z= T$ if and only if $[T]_{i,j}= [T]_{i+1,j+1}$ for every $i,j\geqslant 1.$ This completes the proof.
\end{proof}
In view of Proposition \ref{Toeplitz matrix repn}, the following corollary is now immediate.
\begin{corollary}\label{toeplitz condition satisfies}
For any $\psi\in \calK,$ the Toeplitz operators $T_{\psi}$ satisfies the operator identity 
\begin{align*}
    T_{\bar{z}}T_{\psi}T_z= T_{\psi}.
\end{align*}
\end{corollary}
Let $T\in \mathcal B(\mathcal D_0)$ and $T$ satisfies the identity $T_{\bar{z}}TT_z= T.$ In view of Proposition \ref{BHT implies toeplitz}, it is natural to consider the symbol $\varphi$ given by 
\begin{align*}
    \varphi(z) = \sum\limits_{j=0}^{\infty} c_{-j}\bar{z}^j + \sum\limits_{j=1}^{\infty}c_jz^j,
\end{align*}
where $c_{i-k}= [T]_{i,k}.$ In the following Lemma we find that this $\varphi$ always define a harmonic function on $\mathbb D.$

\begin{lemma}\label{harmonic function existence}
Let $T\in \mathcal B(\mathcal D_0)$ such that the matrix representation of $T$ w.r.t the orthogonal basis $\{z^j:j\in\mathbb Z_{\geqslant 1}\}$ has the \textit{Toeplitz} form 
  \begin{equation*}
  \left[ T \right]=	\begin{pmatrix}
c_0 &  c_{-1} & c_{-2} & \ldots \\
c_1 & c_0  & c_{-1} & \ddots \\
c_2 & c_1 & c_0 & \\
\vdots & \ddots & \ddots & \ddots
	\end{pmatrix},
\end{equation*}
for some sequence of complex numbers $\{c_j\}_{j\in\mathbb Z}.$ Then the function $\varphi$ given by $$\varphi(z)= \sum_{j\geqslant 1}c_jz^j + \sum_{j\geqslant 0}c_{-j}\bar{z}^j,z\in\mathbb D$$ defines a harmonic function on $\mathbb D.$ Moreover the function $\varphi_1(z)=\sum_{j\geqslant 1}c_jz^j\in \mathcal D_0.$
\end{lemma}
\begin{proof}
From the given matrix representation of $T$ with respect to the orthogonal basis $\{z^j:j\in\mathbb Z_{\geqslant 1}\}$ of $\mathcal D_0$ we have that
\begin{align*}
T(z^j)= \sum\limits_{i=1}^{\infty}c_{i-j}z^i,\,\,j\geqslant 1.
\end{align*} 
It follows that 
\begin{align*}
c_{i-j}= \frac{1}{\|z^i\|^2_{\mathcal D_0}}\ip{ T(z^j),z^i}_{\mathcal D_0},\,\,\,i,j\geqslant 1.
\end{align*}
Applying Cauchy-Schwarz inequality we obtain that $|c_{i-j}|\leqslant \|T\| \frac{\sqrt{j}}{\sqrt{i}}$ for $i,j\geqslant 1.$ This gives us that $|c_k| \leqslant \frac{1}{\sqrt{k+1}}\|T\|$ for every $k\geqslant 1$ and $|c_{-j}| \leqslant \sqrt{j+1}  \|T\|$ for every $j\geqslant 1.$ It follows that both the power series $\sum_{j\geqslant 1}c_jz^j$ and $\sum_{j\geqslant 0}c_{-j}\bar{z}^j$ has radius of convergence at least equal to $1.$ Hence $\varphi$ being a sum of holomorphic and antiholomorphic function on $\mathbb D$ must define a harmonic function on $\mathbb D.$ Since $Tz\in \mathcal D_0,$ we have that 
\begin{align*}
\|Tz\|^2_{\mathcal D_0}= \sum\limits_{j=0}^{\infty}|c_j|^2(j+1) <\infty.
\end{align*}
Hence it follows that $\varphi_1(z)=\sum_{j\geqslant 1}c_jz^j\in \mathcal D_0.$ This completes the proof.
\end{proof}
Before proceeding to show that the symbol $\varphi$ induced by $T$ as mentioned in the Lemma \ref{harmonic function existence} belongs to $\calK$ and $T= T_{\varphi},$ we first note down a few basic facts about the decomposition of an operator into homogeneous parts with respect to translations. For each $\theta\in\mathbb [0,2\pi),$ let $\tau_{e^{i\theta}}$ be the translation operator in $\calB \left( \mathcal D_0 \right)$, given by 
\begin{align*}
    (\tau_{e^{i\theta}}f)(z)= f(e^{-i\theta}z),\,\,\,\,z\in\mathbb D.
\end{align*}
Note that  $\tau_{e^{i\theta}}$ is an unitary operator in $\calB \left( \mathcal D_0 \right).$ It follows that for any operator $T\in \calB \left( \mathcal D_0 \right)$,
\begin{equation}
Tf = \lim_{N\to \infty} \sum_{\abs{k} \le N} \left( 1-\frac{\abs{k}}{N+1}  \right) T_{k}f, \qquad f \in \mathcal D_0,
\label{eqn:approx}
\end{equation}
where for each $k\in \mathbb Z,$ the operator $T_k$ denotes the homogeneous part of $T$ of degree $k$ and it is given by 
\begin{align*}
T_kf=\frac{1}{2\pi}\int_{\mathbb T}e^{-ik\theta}\tau_{e^{-i\theta}}T\tau_{e^{i\theta}}(f) d\theta,\,\,f\in \mathcal D_0,
\end{align*}
see \cite[Section I.2]{Katz2004}, see also \cite[Sec 2]{Olof08}.

\begin{proposition}\label{lem:approx-by-hom-parts}
Let $T\in \mathcal B(\mathcal D_0)$ such that the matrix representation of $T$ w.r.t the orthogonal basis $\{z^j:j\in\mathbb Z_{\geqslant 1}\}$ has the \textit{Toeplitz} form 
  \begin{equation*}
  \left[ T \right]=	\begin{pmatrix}
c_0 &  c_{-1} & c_{-2} & \ldots \\
c_1 & c_0  & c_{-1} & \ddots \\
c_2 & c_1 & c_0 & \\
\vdots & \ddots & \ddots & \ddots
	\end{pmatrix},
\end{equation*}
for some sequence of numbers $\{c_j\}_{j\in\mathbb Z}.$ Then it follows that 
\begin{align*}
T_k= \begin{cases}
c_kT_{z^k},& k\geqslant 1\\
c_{-k}T_{\bar{z}^{|k|}}, & k\leqslant 0
\end{cases}, \,\,\mbox{and\,\,\,}\,\,Tf= \lim\limits_{N\to\infty} T_{\sigma_{\!_N}(\varphi)}(f),
\end{align*}
where $\sigma_{\!_N}(\varphi)$ denotes the Ces\'{a}ro sum of the series $\varphi(z)=\sum_{k\geqslant 0} c_{-k}\bar{z}^{k} + \sum_{k\geqslant 1} c_{k}z^{k}$, that is, 
\begin{align*}
\sigma_{\!_N}(\varphi)(z)=   \left( \sum_{k=0}^N \left( 1-\frac{\abs{k}}{N+1}  \right) c_k z^k  + \sum_{k=1}^N \left( 1-\frac{\abs{k}}{N+1}  \right) c_{-k} \bar{z}^{-k} \right).
\end{align*}
\end{proposition}

\begin{proof}
Let $T \in \calB (\calD )$. Let $k \in \mathbb Z$. Notice that for $n,m \ge 1$,
\begin{align*}
    \ip{T_k z^n, z^m} &= \frac{1}{2\pi} \int_{\T } e^{-ik\theta} \ip{\tau_{e^{-i\theta}} T\tau_{e^{i\theta}}z^n , z^m} d\theta \\
        &= \frac{1}{2\pi} \int_{\T } e^{-ik\theta} \ip{T(e^{-in\theta}z ^n) , \tau_{e^{-i\theta}}^*z^m} d\theta \\
        &= \frac{1}{2\pi} \int_{\T } e^{-ik\theta} \ip{T(e^{-in\theta}z ^n) , \tau_{e^{i\theta}}z^m} d\theta  \\
        &= \frac{1}{2\pi} \int_{\T } e^{-ik\theta} \ip{T(e^{-in\theta}z ^n) , e^{-im\theta}z^m} d\theta \\
    &= \frac{1}{2\pi} \int_{\T } e^{i(m-n-k)\theta} \ip{Tz^n , z^m} d\theta \\
    &= \ip{Tz^n, z^m} \delta_{m-n, k}.
\end{align*}
This shows that if the matrix of $T$  with respect to the orthogonal basis $\{ z^k : k \ge 1 \}$is given by $(t_{mn})_{m,n \ge 0}$ then the matrix of $T_k$ with respect to the same basis is given by $(t_{mn} \delta_{m-n,k})_{m,n \ge 0}$. Now, if the matrix of $T$ has a Toeplitz form as mentioned in the statement of the Proposition then it is immediate by Proposition \ref{Toeplitz matrix repn} that 
\begin{align*}
T_k= \begin{cases}
c_kT_{z^k},& k\geqslant 1\\
c_{k}T_{\bar{z}^{|k|}}, & k\leqslant 0.
\end{cases}
\end{align*}
Consequently, in view of \ref{eqn:approx}, it follows that
\begin{equation*}
Tf = \lim_{N\to \infty} \left( \sum_{k=0}^N \left( 1-\frac{\abs{k}}{N+1}  \right) c_k T_{z^k}f  + \sum_{k=1}^N \left( 1-\frac{\abs{k}}{N+1}  \right) c_{-k} T_{\overline{z}^k}f \right), \qquad f \in \mathcal D_0.
\end{equation*}
Note that $T_{g+h}(f)=T_g(f)+T_h(f)$ for every $g,h\in \mathcal T(\calD)$ and $f\in\mathcal D_0.$ It follows that
\begin{align*}
   T_{\sigma_{\!_N}(\varphi)}(f)=  \left( \sum_{k=0}^N \left( 1-\frac{\abs{k}}{N+1}  \right) c_k T_{z^k}f  + \sum_{k=1}^N \left( 1-\frac{\abs{k}}{N+1}  \right) c_{-k} T_{\overline{z}^k}f \right), \qquad f \in \mathcal D_0.
\end{align*}
Thus we obtain that $\lim Tf= \lim\limits_{N\to\infty} T_{\sigma_{\!_N}(\varphi)}(f)$ for every $f\in\mathcal D_0.$
\end{proof}
Now we are ready to give the proof of the main theorem of this section. 
\begin{proof}[Proof of the Theorem \ref{thm:BHT theorem for D}]
 Let $T\in \mathcal B(\mathcal D_0)$ such that the identity $T_{\bar{z}}TT_z= T$ holds. Then by Proposition \ref{BHT implies toeplitz}, we find that matrix representation of $T$ w.r.t. the orthogonal basis $\{z^j:j\in \mathbb Z_{\geqslant 1}\}$ of $\mathcal D_0$ has the \textit{Toeplitz} form 
  \begin{equation*}
  \left[ T \right]=	\begin{pmatrix}
c_0 &  c_{-1} & c_{-2} & \ldots \\
c_1 & c_0  & c_{-1} & \ddots \\
c_2 & c_1 & c_0 & \\
\vdots & \ddots & \ddots & \ddots
	\end{pmatrix},
\end{equation*}
for some sequence of numbers $\{c_j\}_{j\in\mathbb Z}.$ Moreover in view of Lemma \ref{harmonic function existence}, the function $\varphi$ given by 
\begin{align*}
\varphi(z)= \sum_{j\geqslant 1}c_jz^j + \sum_{j\geqslant 0}c_{-j}\bar{z}^j,z\in\mathbb D,
\end{align*}
defines a harmonic function on $\mathbb D.$ We write $\varphi_1(z)= \sum_{j\geqslant 1}c_jz^j.$ By Lemma \ref{harmonic function existence}, it also follows that $\varphi_1\in \mathcal D_0$. 
For each $w\in\mathbb D,$ we consider 
\begin{align*}
E_w(z)= \frac{z}{1-\bar{w}z} = \sum\limits_{j=0}^{\infty}z^{j+1}\bar{w}^j, \qquad z\in \mathbb D
\end{align*}
 It is straightforward to verify that $\|E_w\|^2_{\mathcal D_0}= \frac{1}{(1-|w|^2)^2}.$ Now we consider 
 \begin{align*}
 \tilde{E}_w(z) = \frac{E_w(z)}{\|E_w\|}= (1-|w|^2)\frac{z}{1-\bar{w}z},\qquad z\in \mathbb D.
 \end{align*}
Note that $\frac{\dou E_{w} (z)}{\dou z}= \tilde{K^B_w}(z),$ the normalised reproducing kernel of $\calA$ at $w\in\mathbb D.$ Using \eqref{lem:inner-product} and the reproducing property of $K^B_w,$ it is immediate that for $k\ge 0$,
\begin{align*}
  \ip{T_{z^k} \tilde{E}_w, \tilde{E}_w} = k(1-\abs{w}^2)w^{k} + w^k \,\,\mbox{and\,\,}\ip{T_{\overline{z}^k} \tilde{E}_w, \tilde{E}_w} = \overline{w}^k,\,\,w\in\mathbb D.
\end{align*}
Now, fix $N \in \mathbb N, w\in \D$ and Observe that 
\begin{align*}
    \ip{ \sum_{k=0}^N \left( 1-\frac{\abs{k}}{N+1}  \right) c_k T_{z^k} \tilde{E}_w, \tilde{E}_w} &= \sum_{k=0}^N \left( 1-\frac{\abs{k}}{N+1}  \right) c_k \left( k(1-\abs{w}^2)w^{k} + w^k \right) 
\end{align*}

and similarly, we have 
\begin{align*}
    \ip{ \sum_{k=1}^N \left( 1-\frac{\abs{k}}{N+1}  \right) c_k T_{\bar{z}^k} \tilde{E}_w, \tilde{E}_w}
    &= \sum_{k=1}^N \left( 1-\frac{\abs{k}}{N+1}  \right) c_k \overline{w}^k.
\end{align*}
This gives us that for any $w\in\mathbb D,$
\begin{align*}
    \ip{T_{\sigma_{\!_N}(\varphi)}\tilde{E}_w, \tilde{E}_w} &= w(1-\abs{w}^2)\sum_{k=0}^N \left( 1-\frac{\abs{k}}{N+1}  \right) c_k kw^{k-1} + \sigma_{\!_N}(\varphi)(w)\\
    &=w(1-\abs{w}^2)\sigma_{\!_N}(\varphi_1')(w) + \sigma_{\!_N}(\varphi)(w).
\end{align*}
Now using Proposition \ref{lem:approx-by-hom-parts} we have that for any $w\in\mathbb D,$
\begin{align}\label{Berezin action formula}
\ip{T \tilde{E}_w, \tilde{E}_w} &=  \lim_{N\to\infty} \ip{T_{\sigma_{\!_N}(\varphi)}\tilde{E}_w, \tilde{E}_w} \notag\\
&=   w(1-\abs{w}^2) \lim_{N\to\infty}\sigma_{\!_N}(\varphi_1')(w) +  \lim_{N\to\infty}\sigma_{\!_N}(\varphi)(w), \notag\\
&= w(1-|w|^2) \varphi_1'(w) + \varphi(w).
\end{align}
It is well known that $\mathcal D_0$ is contained in $\mathcal B_0,$ the little Bloch space given by
\begin{align*}
\mathcal B_0=\Big\{ f\in\mathcal O(\mathbb D): \lim\limits_{|z|\to 1}(1-|z|^2)|f'(z)|=0\Big\},
\end{align*}
see \cite[Chapter 5]{Zhu2007}. So we find that $\lim\limits_{|z|\to 1}(1-|z|^2)|\varphi_1'(z)|=0.$ Since $\tilde{E}_w$ is a unit vector, we have that $|w(1-|w|^2) \varphi_1'(w) + \varphi(w)| \leqslant \|T\|$ for every $w\in\mathbb D.$ It follows that $\varphi$ is bounded on $\mathbb D$ and 
\begin{align}\label{sup norm vs operator norm}
   \|\varphi\|_{\infty} \leqslant \|T\|.
\end{align}
Thus we obtain that 
$\varphi\in \tilde{\mathcal T}(\calD).$ Now consider the densely defined Toeplitz operator $T_{\varphi}$ on $\mathcal D_0.$ Note that
\begin{align*}
\ip{T_{\varphi}z^n,z^m}_{\mathcal D_0}= \ip{Tz^n,z^m}_{\mathcal D_0},\,\,\,n,m\geqslant 1.
\end{align*}
Since $T$ is a bounded operator and span$\{z^j:j\in \mathbb Z_{\geqslant 1}\}$ is dense in $\mathcal D_0,$ it follows that $T_{\varphi}$ extends to a bounded operator on $\mathcal D_0$ and $T_{\varphi}=T.$ In view of Theorem \ref{thm:bdd-top-op}, it follows that $\varphi \in \mathcal T(\calD).$ This completes the proof.
\end{proof}
 
\begin{corollary}\label{cor:WOT-closed}
The subspace $\{ T_\varphi : \varphi \in \calK \} \subseteq \calB (\calD)$ is closed in the weak operator topology (W.O.T) of $\calB (\calD)$.
\end{corollary}
\begin{proof}
    Let $X,Y\in \mathcal B(\calD).$ A standard argument shows that the subspace $S_{X,Y}$ of $\mathcal B(\calD)$ given by  $S_{X,Y}=\{A\in \mathcal B(\calD) : XAY=A\}$ is closed in the weak operator topology in $\mathcal B(\calD).$ Thus the corollary follows immediately in view of Theorem \ref{thm:BHT theorem for D}.
\end{proof}

\begin{corollary}
For any $\varphi\in  \mathcal T(\calD),$ the sequence of Toeplitz operators $\{T_{\sigma_{\!_N}(\varphi)}\}_{N\geqslant 1}$ converges to $T_{\varphi}$ in the strong operator topology in $\calB(\calD).$ 
\end{corollary}
\begin{proof}
    The proof follows immediately in view of Proposition \ref{lem:approx-by-hom-parts}.
\end{proof}

\section{Some algebraic properties of Toeplitz operators}\label{Sec 4}
In this section, we investigate some algebraic properties of the Toeplitz operators on the Dirichlet space $\calD.$  Considering the various studies by different authors on Toeplitz operator in Dirichlet space $\calD$ induced by several symbol classes, the results in this section may be expected but these still are interesting facts. In \cite{Lee2007}, Y. J. Lee studied various algebraic properties  for the class of Toeplitz operators  $\{T_{\varphi} : \varphi\in \Omega \cap h(\D)\},$ where the symbol class $\Omega$ is given by $\Omega=\{ \varphi \in \mathscr C^1(\mathbb D):  \varphi, \frac{\dou \varphi }{\dou z}, \frac{\dou \varphi }{\dou \bar z} \in L^{\infty}(\mathbb D)\}$ and $h(\D)$ denotes the set of all complex valued harmonic functions on $\D$. It turns out that all those characterizations obtained by Lee in \cite{Lee2007}, naturally get extended to the Toeplitz operators  $\{T_{\varphi} : \varphi\in \mathcal T(\calD) \},$  acting on $\calD.$  Most of the methods and techniques used by Lee can be used to extend his result for the Toeplitz operators $T_{\varphi} $ induced by the symbol class $\mathcal T(\calD).$   But rather we chose to exploit the Brown-Halmos type operator identity (wherever possible), to obtain the characterizations of these algebraic properties of the Toeplitz operators  acting on $\calD.$ First we show that product of two Toeplitz operator is Toeplitz if and only if  the first symbol is antiholomorphic or the second symbol is holomorphic. We first prove a couple of preliminary lemmas in this regard. The first lemma, in the case of Toeplitz operators on the Hardy space induced by $L^{\infty}(\mathbb T)$ symbols, is well known, see for example \cite[Chapter 3]{Rosenthalbook} and references therein. 

\begin{lemma} \label{lem:rank-one}
    Let $\psi, \varphi \in \mathcal T(\calD)$. Then
    \begin{equation*}
        T_{\overline{z}} T_\psi T_\varphi T_z = T_\psi T_\varphi + \left( T_{\overline{z}} T_{\psi} z \right) \otimes \left( T_{z}^* T_\varphi ^* z \right).
    \end{equation*}
    \end{lemma}
    \begin{proof}
        It is straightforward to verify that $I = T_{z} T_{\bar{z}} + z \otimes z$. Thus, we have
        \begin{align*}
            T_{\overline{z}} T_\psi T_\varphi T_z &= T_{\overline{z}} T_\psi ( T_{z} T_{\bar{z}} + z \otimes z ) T_\varphi T_z \\
            &= T_\psi T_\varphi + T_{\overline{z}} T_\psi (z\otimes z) T_{\varphi} T_z \\
            &= T_\psi T_\varphi + \left( T_{\overline{z}} T_{\psi} z \right) \otimes \left( T_{z}^* T_\varphi ^* z \right).
        \end{align*}
        This completes the proof of the lemma.
    \end{proof}

\begin{lemma}
    Let $\psi, \varphi \in  \mathcal T(\calD)$. If $\psi$ is antiholomorphic or $\varphi$ is holomorphic then 
    \begin{equation*}
        \ip{T_{\psi} T_{\varphi}f, g} = \ip{\frac{\dou (\psi \varphi f) }{\dou z} , g'}_{L^2 (\D)},\qquad f,g\in\mathcal D_0.
    \end{equation*}
    \label{lem:nature-of-anti-hol}
\end{lemma}

\begin{proof}
First assume that $\varphi$ holomorphic. Then for $f,g \in \calD,$  using \eqref{lem:inner-product}, we have that 
\begin{align*}
    \ip{T_{\psi} T_{\varphi} f, g} = \ip{\left( \parz {\psi} \right)  (\varphi f) + \psi (\varphi f)', g}_{L^2 (\D)} = \ip{\parz {(\psi \varphi f)}, g'}_{L^2 (\D)} \stepcounter{equation}\tag{\theequation}\label{eqn:hol}
\end{align*}
   Now assume that $\overline{\psi} =\chi$ is a holomorphic symbol.  Then for $f,g \in \calD$, using \eqref{action of densely defined toeplitz op} and \eqref{lem:inner-product}, we have that
    \begin{align*}
        \ip{T_{\psi}T_{\varphi}f, g} &= \ip{\frac{\dou  (\overline{\chi} T_{\varphi} f)}{\dou z}, g'}_{L^2 (\D)} \\
        &= \ip{\overline{\chi} \left( f \frac{\dou \varphi }{\dou z} + P_{\calA} (\varphi f') \right), g'}_{L^2 (\D)} \\
        &= \ip{f  \frac{\dou \varphi }{\dou z} + P_{\calA} (\varphi f'), \chi g'}_{L^2 (\D)} \\
        &= \ip{f  \frac{\dou \varphi }{\dou z} , \chi g'}_{L^2 (\D)} + \ip{(\varphi f'), \chi g'}_{L^2 (\D)} \\
        &=\ip{\overline{\chi} \left( f  \frac{\dou \varphi }{\dou z} + \varphi f' \right), g'}_{L^2 (\D)} \\
        &= \ip{ \frac{\dou \left( \psi \varphi f \right) }{\dou z}  , g'}_{L^2 (\D)}
\stepcounter{equation}\tag{\theequation}\label{eqn:antihol}\\
    \end{align*}
This completes the proof.
\end{proof}


Now we are ready to answer the question of when product of two Toeplitz operator is again a Toeplitz operator.
\begin{proposition}\label{Product Toeplitz}
    Let $\psi, \varphi \in \mathcal T(\calD)$. Then  $T_\psi T_\varphi = T_\tau$ for some symbol $\tau\in \mathcal T(\calD)$ if and only if $\psi$ is antiholomorphic or $\varphi$ is holomorphic. Moreover in either case we must have $\tau= P[(\psi \varphi)|_{_\mathbb T}].$ 
\end{proposition}
\begin{proof}
In view of Corollary \ref{toeplitz condition satisfies} and  Lemma \ref{lem:rank-one}, it follows that  the product $T_\psi T_\varphi = T_\tau$ for some $\tau\in\mathcal T(\calD)$ if and only if 
\begin{equation*}
    \left( T_{\overline{z}} T_{\psi} z \right) \otimes \left( T_{z}^* T_\varphi ^* z \right) = 0.
\end{equation*}
This implies that either $T_{\overline{z}} T_{\psi} z = 0$ or $T_{z}^* T_\varphi ^* z = 0$. A routine calculation using Proposition \ref{Toeplitz matrix repn}  shows that this happens if and only if either $\psi$ is antiholomorphic or $\varphi$ is holomorphic. Now, we show that $(\psi \phi)|_{_{\mathbb T}} = \tau|_{_{\mathbb T}}.$ By Lemma \ref{lem:nature-of-anti-hol}, we have 
\begin{align*}
    \ip{T_{\psi} T_{\varphi} \tilde{E}_w, \tilde{E}_w}
    &= \ip{\parz {(\psi \varphi \tilde{E}_w)}, \tilde{E}_w'}_{L^2 (\D)} \\
    &= \ip{\parz {\varphi} \tilde{E}_w, \tilde{K_w ^B}}_{L^2 (\D)} + \ip{\psi \varphi \tilde{K_w ^B}, \tilde{K_w ^B}}_{L^2 (\D)}  \\
    &= w(1-\abs{w}^2) \parz{ \varphi }(w) +\ip{\psi \varphi \tilde{K_w ^B}, \tilde{K_w ^B}}_{L^2 (\D)} .
\end{align*}
and a similar argument shows that 
\begin{align*}
    \ip{T_{\tau} \tilde{E}_w, \tilde{E}_w} &= w(1-\abs{w}^2) \parz \tau (w) + \tau (w).
\end{align*}
Since $\psi$, $\varphi$ are bounded harmonic functions, $\ip{\psi \varphi \tilde{K_w ^B}, \tilde{K_w ^B}}_{L^2 (\D)}$ must have non-tangential limits almost everywhere on $\T$, see \cite[Exercise 2, Chapter 2]{HKZ2012}. Now, taking radial limits as $\abs{w} \to 1$ in the two aforementioned equalities, we have that $(\psi \phi)|_{_{\mathbb T}} = \tau|_{_{\mathbb T}}.$ Since $\tau$ is a bounded harmonic function we must have $$\tau= P[\tau|_{_{\mathbb T}}] =P[(\psi \phi)|_{_{\mathbb T}}].$$ This completes the proof.
\end{proof}
Now we discuss the problem of characterizing two symbols for which the corresponding Toeplitz operators commute on $\calD.$

 \begin{proposition}
    Let $\psi , \varphi \in \mathcal T(\calD)$. Write
        \begin{equation*}
            \varphi (z) = \sum_{k\ge 0} \hat{\varphi} (k) z^k + \sum_{k \ge 1} \hat{\varphi} (-k) \overline{z}^k, \qquad z\in \D,
        \end{equation*}
        and
        \begin{equation*}
            \psi (z) = \sum_{k\ge 0} \hat{\psi} (k) z^k + \sum_{k \ge 1} \hat{\psi} (-k) \overline{z}^k,\qquad z\in \D.
        \end{equation*}
        The following are equivalent:
        \begin{enumerate}
            \item $T_{\varphi} T_{\psi} = T_{\psi} T_{\varphi}$ on $\calD$.
            \item $\hat{\varphi} (m) \hat{\psi} (-n) = \hat{\varphi} (-m) \hat{\psi} (n)$ for $m,n \ge 1$.
            \item Exactly one of the following is true:
            \begin{enumerate}
                \item $\varphi$ and $\psi$ is holomorphic.
                \item $\varphi$ and $\psi$ is antiholomorphic.
                \item A nontrivial linear combination of $\varphi$ and $\psi$ is constant in $\D$.
            \end{enumerate}
        \end{enumerate}
    \end{proposition}
\begin{proof}
        We first show that $(1) \Longleftrightarrow (2)$. Using Lemma \ref{lem:rank-one} and Proposition \ref{Toeplitz matrix repn}, we find that 
\begin{align*}
            T_{\varphi} T_{\psi} = T_{\psi} T_{\varphi} & \Longleftrightarrow \left( T_{\overline{z}} T_{\varphi} z \right) \otimes \left( T_{z}^* T_\psi ^* z \right) =  \left( T_{\overline{z}} T_{\psi} z \right) \otimes \left( T_{z}^* T_\varphi ^* z \right) \\
            &\Longleftrightarrow \ip{T_{\psi} z^{n+1}, z} T_{\overline{z}} T_{\varphi} z = \ip{T_{\varphi} z^{n+1}, z} T_{\overline{z}} T_{\psi} z,  &\qquad n \ge 1, \\
            &\Longleftrightarrow \hat{\varphi} (m) \hat{\psi} (-n) = \hat{\varphi} (-m) \hat{\psi} (n), &\qquad m,n \ge 1.
        \end{align*}
        The fact $(2) \Longleftrightarrow (3)$ is routine and well studied. We leave it for the reader to verify, see \cite{BHT63,Lee2007} for similar arguments.
    \end{proof}
We conclude the paper with a brief discussion on the question of  when the product of two Toeplitz operators on $\calD$ is a compact perturbation of some Toeplitz operator on $\calD$. For any $f\in L^1(\mathbb D),$ let ${\bf B}(f)$ denotes the Berezin transform of $f,$ see \cite[Chapter 2]{HKZ2012} for the definitions and the properties of the Berezin transform and $C_0$ denotes the space of all continuous functions $g$ on $\D$ such that $g(a)\to 0$ as $|a|\to 1-.$ We find that the result of Y. J. Lee in \cite{Lee2007} on the compact product problem of when $T_{\psi}T_{\varphi} - T_{\tau}$ is compact on $\mathcal D_0$ for $\varphi,\psi,\tau\in \Omega \cap h(\D)$ naturally extends for the symbols $\varphi,\psi,\tau\in \mathcal T (\calD).$ We state the result in the next proposition and omit the proof here. One may imitate the argument presented in \cite[Theorem D]{Lee2007} to obtain the following:
\begin{proposition}\label{compact product}
    Let  $\psi, \varphi, \tau \in \mathcal T(\calD)$. Then the operator $T_{\psi}T_{\varphi} - T_{\tau}$ is a compact operator if and only ${\bf B}({\psi \varphi})-\tau\in C_0.$
\end{proposition}
     \part*{\textbf{Declarations}}
\noindent Conflict of interest : The authors declare that there are no conflict of interest regarding the publication of this paper.

\part*{\textbf {Acknowledgments}}
\noindent The authors wish to express their sincere thanks to Prof. Sameer Chavan for his several valuable comments and suggestions.



\begin{thebibliography}{55}

 \bibitem{Adams75}
R. A. Adams, \textit{Sobolev Spaces}, Pure and Applied Mathematics {\bf{65}}, Academic Press, New York, 1975.

\bibitem{AVZP2009} 
A. Aleman and D. Vukoti\'{c}, \textit{Zero products of Toeplitz operators} (2009): 373-403.

\bibitem{Axler91}
S. Axler and \v{Z}. \v{C}u\v{c}kovi\'{c}, \textit{Commuting Toeplitz operators with harmonic symbols,} Integral Equations Operator Theory
{\bf{14}} (1991), 1–11.



\bibitem{BHT63}
A. Brown and P.R. Halmos, \textit{Algebraic properties of Toeplitz operators}, J. Reine Angew. Math. {\bf 213}, 1964 89–102.

\bibitem{Chen09}
Y. Chen, \textit{Commuting Toeplitz operators on the Dirichlet space,}  J. Math. Anal. Appl. {\bf 357} (2009), no. 1, 214-224.



\bibitem{ChenNguyen10}
 Y. Chen and  Q. D. Nguyen, \textit{Toeplitz and Hankel operators with $L^{\infty,1}$ symbols on Dirichlet space}, Journal of Mathematical Analysis and Applications {\bf 369.1} (2010): 368-376.

 \bibitem{Choe99}
B. R. Choe and Y. J. Lee, \textit{Commuting Toeplitz operators on the harmonic Bergman spaces}, Michigan Math J. {\bf{46}} (1999): 163-174.



 \bibitem{Cuko94}
\v{Z}. \v{C}u\v{c}kovi\'{c}, \textit{Commutants of Toeplitz operators on the Bergman spaces,} Pacific J. Math.
{\bf{162}} (1994), 277-285.

\bibitem{DuiLee2004}
J. J. Duistermaat and Y. J. Lee, \textit{Toeplitz operators on the Dirichlet space}, Journal of Mathematical Analysis and Applications {\bf{300.1}} (2004): 54-67.

\bibitem{Primer}
O.~El-Fallah, K.~ Kellay, J.~ Mashreghi and Thomas Ransford, \emph{A primer on the Dirichlet space} {\bf{ 203}}, (2014),~ Cambridge University Press.

\bibitem{Englis88}
M. Engli\v{s},  \textit{A note on Toeplitz operators on Bergman spaces}, Comm. Math. Univ. Carolinae, {\bf{29}} (1988): 217-219.

\bibitem{Eschmier2019}
J. Eschmeier and S. Langend\"{o}rfer, \textit{Toeplitz operators with pluriharmonic symbol on the unit ball}, Bull. Sci. Math. {\bf{151}} (2019): 34-50.


\bibitem{HKZ2012}
H. Hedenmalm , B. Korenblum and K. Zhu, \textit{Theory of Bergman spaces},  Vol.{\bf 199}, Springer Science \& Business Media, 2012.


\bibitem{Katz2004}
Y. Katznelson, \textit{An introduction to harmonic analysis}. Cambridge University Press (2004).
 
 \bibitem{Lee2007}
 Y. J. Lee, \textit{Algebraic properties of Toeplitz operators on the Dirichlet space}, Journal of Mathematical Analysis and Applications {\bf 329} (2007), no. 2, 1316-1329.
 
 \bibitem{Lee2009}
 Y. J. Lee, \textit{Finite sums of Toeplitz products on the Dirichlet space}, Journal of Mathematical Analysis and Applications {\bf 357.2} (2009): 504-515.
 
 
 \bibitem{LeeZhu2011}
Y. J. Lee and  K. Zhu, \textit{Sums of products of Toeplitz and Hankel operators on the Dirichlet space}, Integral equations and operator Theory {\bf 71} (2011): 275-302.


\bibitem{Lin2018}
H. Lin, Z. Zhang and D. Zheng, \textit{Spectrum, Asymptotic Invertibility and Szeg\"{o} Type Theorems of Dirichlet Toeplitz Operators,} Integral Equations Operator Theory, {\bf{90}} (2018), no.4, Paper No. 44, 24 pp.

\bibitem{Olof08}
I. Louhichi and A. Olofsson, \textit{Characterizations of Bergman space Toeplitz operators with harmonic symbols}, J. Reine Angew. Math. {\bf 617} (2008): 1-26.


\bibitem{LuoXiao2018}
S. Luo and J. Xiao, \textit{Toeplitz Operators for Dirichlet Space Through Sobolev Multiplier Algebra,} Taiwanese Journal of Mathematics, {\bf 22} (2018), no. 2, 383-399.

\bibitem{Rosenthalbook}
R.~ A.~ Mart\'{i}nez-Avenda\~{n}o, P. Rosenthal, \textit{An introduction to operators on the Hardy-Hilbert space,} Vol. 237. New York: Springer, 2007.

\bibitem{Antti2017}
A. Per\"{a}l\"{a}, J. Taskinen and J. A. Virtanen, \textit{Toeplitz operators on Dirichlet-Besov spaces,} Houston J. Math. {\bf{43}} (2017), no.1, 95–110.


 \bibitem{RochbergWu1992}
R. Rochberg and  Z. J. Wu, \textit{Toeplitz operators on Dirichlet spaces}, Integral equations and operator Theory {\bf 15} (1992), no. 2, 325-342.

\bibitem{Wu1992}
Z. J. Wu, \textit{Hankel and Toeplitz operators on Dirichlet spaces}, Integral equations and operator Theory {\bf 15} (1992), no. 3, 503-525.

\bibitem{Yu2010}
T. Yu, \textit{Toeplitz operators on the Dirichlet space}, Integral equations and operator Theory {\bf 67} (2010), no. 2, 163-170.

\bibitem{Zhu2007}
K. Zhu,  \textit{Operator theory in function spaces} (2007), No. 138. American Mathematical Soc. 



\end{thebibliography}
\end{document}